\documentclass[reqno, 10pt]{amsart}
\usepackage[
 	hmarginratio={1:1},     
 	vmarginratio={1:1},     
	textheight=600pt,
	textwidth=400pt,        
 	heightrounded,          
]{geometry}
\usepackage{amssymb}
\usepackage{hyperref}
\usepackage{tikz-cd}

\title[Extensions of the representation ring global functor and $K$-theory]{Rational extensions of the representation ring global functor and a splitting of global equivariant $K$-theory}
\author{Christian Wimmer}
\date{\today} 
\address{Mathematisches Institut, Universit\"at Bonn, Endenicher Allee 60, 53115 Bonn, Germany}
\email{wimmer@math.uni-bonn.de}
\keywords{Global functors, representation ring, equivariant $K$-theory, Chern character}
\subjclass[2010]{19A22, 20C05, 55P91}
\newcommand{\bbA}{\mathbb A}
\newcommand{\bbC}{\mathbb C}
\newcommand{\bbN}{\mathbb N}
\newcommand{\bbQ}{\mathbb Q}
\newcommand{\bbZ}{\mathbb Z}

\newcommand{\lra}{\longrightarrow}
\newcommand{\ra}{\rightarrow}

\newcommand{\RU}{\mathbf{RU}}
\newcommand{\RUrat}{\mathbf{RU}_{\bbQ}}
\newcommand{\RO}{\mathbf{RO}}

\newcommand{\tensor}{\otimes}

\newcommand{\Outop}{\Out^{\op}}
\newcommand{\Outopmod}{\Out^{\op}\hyph\module}
\newcommand{\Outopcyc}{\Out^{\op}_{\cyc}}
\newcommand{\Outopcycp}{\Out^{\op}_{\cyc\hyph p}}

\newcommand{\GF}{\mathcal{GF}}

\DeclareMathOperator{\Aut}{Aut}
\DeclareMathOperator{\cyc}{cyc}
\DeclareMathOperator{\Ext}{Ext}

\DeclareMathOperator{\fin}{fin}
\DeclareMathOperator{\Hom}{Hom}
\DeclareMathOperator{\hyph}{-}
\DeclareMathOperator{\Id}{Id}
\DeclareMathOperator{\module}{mod}
\DeclareMathOperator{\Out}{Out}
\DeclareMathOperator{\op}{op}
\DeclareMathOperator{\res}{res}
\DeclareMathOperator{\tr}{tr}

\DeclareMathOperator{\KU}{KU}

\numberwithin{equation}{section}
\newtheorem{Corollary}[equation]{Corollary}
\newtheorem{Lemma}[equation]{Lemma}
\newtheorem{Proposition}[equation]{Proposition}
\newtheorem{Theorem}[equation]{Theorem}
\newtheorem{Question}[equation]{Question}

\theoremstyle{definition}
\newtheorem{Definition}[equation]{Definition}
\newtheorem{Example}[equation]{Example}

\theoremstyle{remark}
\newtheorem{Remark}[equation]{Remark}

\begin{document}

\maketitle

\begin{abstract}
We identify the group of homomorphisms $\Hom_{\GF}(F,\RUrat)$ in the category of ($\fin$)-global functors to the rationalization of the unitary representation ring functor
and deduce that the higher $\Ext$-groups $\Ext^n_{\GF}(F,\RUrat)$, $n\geq 2$ have to vanish.
This leads to a rational splitting of the ($\fin$)-global equivariant $K$-theory spectrum into a sum of Eilenberg-MacLane spectra.
Interpreted in terms of cohomology theories, it means that the equivariant Chern character is compatible with restrictions along all group homomorphisms.
\end{abstract}

\setcounter{tocdepth}{1}
\tableofcontents

\section{Introduction}
Globally defined Mackey functors have emerged in the last two decades as a naturally occuring algebraic structure with applications in group cohomology and representation theory
\cite{Webb}, \cite{guidemf}, \cite{bouc1}, \cite{bouc2}, \cite{webbstrat}.
\emph{Global functors} are a particular version (called inflation functors in \cite{Webb}) appearing in equivariant topology,
where the 'global' perspective has also proven to be illuminating in the study of equivariant homotopy types \cite{globalsymmetric}, \cite{hausi}, \cite{hausmann}.
They are defined more generally for compact Lie groups, but we will restrict attention to finite groups in the following.
Roughly speaking, a ($\fin$)-global functor $M$ consists of abelian groups $M(G)$ for every finite group $G$,
which are related by certain transfer (or induction) maps along subgroup inclusions and restrictions along all group homomorphisms, cf.\ Definition \ref{def:globalfunctor}. 
A global functor in particular encodes compatible Mackey functors for each group $G$, but the restrictions along surjective morphisms, the \emph{inflations}, form an additional datum.

This paper is concerned with the computation of certain $\Ext$-groups in the category of global functors
with respect to finite groups, see Theorem \ref{thm:datheorem} below. They were motivated by the following question in (global) equivariant stable homotopy theory:
\begin{Question}
Does the global $K$-theory spectrum split rationally ?
\end{Question}

Historically, topological $K$-theory \cite{atiyahhirzebruch} was one of the first generalized cohomology theories considered. 
The Chern character is a natural transformation (say on finite complexes)
\[
	\operatorname{Ch}:K(X)\lra \bigoplus_{n\geq 0} H^{2n}(X,\bbQ),
\]
which gives a rational isomorphism with the singular homology groups in even degrees.
From the point of view of stable homotopy theory, this is a consequence of the rational splitting of the stable homotopy category $\mathcal{SHC}$:
Taking homotopy groups induces an equivalence
\(
	\mathcal{SHC}_{\bbQ}\simeq \operatorname{gr.} \bbQ\hyph\module
\)
with the category of integer-graded rational vector spaces. So every spectrum with rational homotopy groups decomposes into a sum of Eilenberg-MacLane spectra
(i.e.\ those with homotopy groups concentrated in a single degree); in particular, we have a decomposition
\[
	\KU_{\bbQ}\simeq \bigvee_{n\in \bbZ}\Sigma^{2n}H\bbQ
\]
for the spectrum $\KU$ representing $K$-theory.

In the classical equivariant setting, the situation is similar.
For a finite group $G$, there is a natural generalization of $K$-theory to an '${RO}(G)$-graded' theory $K_G$ (\cite{segal}, \cite[XIV]{alaska}),
the most sophisticated version of cohomology theories on $G$-spaces usually considered.
It is represented by a spectrum $\KU_G$ in the genuine $G$-equivariant stable homotopy category, which also splits rationally (cf.\ \cite[Appendix A]{tate}).
In terms of cohomology theories this means that there is an equivariant Chern character, compatible with the natural structure maps of these theories,
e.g.\ restriction and transfer maps relating the values at various subgroups of $G$.

But he story is not quite over yet: as the group $G$ varies, the equivariant homotopy types $\KU_G$ do not just form some collection. 
They are part of a \emph{global equivariant} homotopy type $\KU$, which can be thought of as encoding a spectrum valued global functor.
In global equivariant homotopy theory one systematically studies these and the resulting analogue of the stable homotopy category is
the \emph{global stable homotopy category} $\mathcal{GH}$. Rephrasing the above,
there is a global $K$-theory spectrum $\KU\in \mathcal{GH}$
that represents the equivariant cohomology theories $K_G$ \cite[6.4]{global}.
The equivariant homotopy groups $\pi_k^GE\cong E_G^{-k}(\ast)$ of a global stable homotopy type $E$
are the coefficients of the associated $G$-equivariant cohomology theory and as $G$ runs over all finite groups
these naturally carry the structure of a global functor. In the case of $K$-theory we obtain the representation ring:
\[
	\pi_0^G \KU\cong K_G(\ast)\cong \RU(G).
\]

In contrast to the classical setting, rational global homotopy types in general do not decompose into global Eilenberg-MacLane spectra:
The rational $(\fin)$-global homotopy category is algebraic in the sense that there is an equivalence
\[
	\mathcal{GH}_{\bbQ}\simeq \mathcal D(\GF_{\bbQ})
\]
with the derived category of global functors under which taking homotopy corresponds to taking homology \cite[Thm 4.5.29]{global}.
However, even rationally the homological dimension is infinite (see Remark \ref{remark:infinitehomological}) and so homology does not yield a faithful functor to the graded category.
But it turns out that the appropriate $\Ext$-groups vanish in the special case of $K$-theory (Theorem \ref{thm:daktheorythm}).

This is based on the following calculation. For each integer $n\geq 1$, let $C_n$ be a cyclic group of order $n$. If $F$ is a global functor,
then $F(C_n)^{\vee}_{\tr}$ denotes the $\bbQ$-linear forms on $F(C_n)$ that vanish on the image of all transfers from proper subgroups.
These can be organized into an inverse system, cf.\ Definition \ref{def:inversesystem}.

\begin{Theorem}\label{thm:datheorem}
There is a preferred natural isomorphism
\[
	\Hom_{\GF_\bbQ}(F,\RUrat)\cong {\varprojlim}_{n\in(\bbN,\mid)}(F(C_n)^{\vee}_{\tr})
\]
identifying maps into $\RUrat$ as an inverse limit over the poset of natural numbers with respect to the divisibility relation.
Moreover, it also induces isomorphisms 
\[
	\Ext^\ast_{\GF_\bbQ}(F,\RUrat)\cong {\varprojlim}_{n\in(\bbN,\mid)}^\ast (F(C_n)^{\vee}_{\tr})
\]
between higher derived functors.
\end{Theorem}
\begin{Remark}
The universal linear forms $\phi_n:\RU(C_n)\otimes\bbQ\ra \bbQ$ are determined by an explicit choice of compatible bases
for the representation rings of cyclic $p$-groups (Proposition \ref{prop:outru}) and the multiplicative relation
$\phi_{nm}=\phi_n\cdot\phi_m$ for coprime integers $n$ and $m$ (Lemma \ref{lemma:outassemble}).

\end{Remark}

By cofinality we will deduce the following result:
\begin{Corollary}\label{cor:dacorollary}
For any global functor $F$ the higher $\Ext$-groups 
\[
	\Ext^\ast_{\GF}(F,\RUrat)=0, \quad n\geq 2
\]
vanish. Addtionally, the $\Ext$-algebra of $\RUrat$ vanishes in all positive degrees:
\[
	\Ext^1_{\GF}(\RUrat,\RUrat)=0.
\]
\end{Corollary}
\begin{Remark}
The same is true for the real representation ring $\RO$ because it is a rational retract of $\RU$.
\end{Remark}
We briefly return to the original question. As mentioned above, there is a rational equivalence $\mathcal{GH}_\bbQ\simeq \mathcal D(\GF_\bbQ)$ between the global homotopy category and the derived category of global functors.
The homotopy groups of $\KU$ are isomorphic to the representation ring global functor $\RU$ in even degrees and trivial in odd degrees.
Since all $\Ext$-groups between these vanish in cohomological degrees larger than two, the associated chain complex is quasi-isomorphic to its homology.
Additionally, the first $\Ext$-groups are trivial and so there are no 'exotic' endomorphisms; that is, the selfmaps in the derived category biject with the graded morphisms between the homology groups.
Translating back to the topological side:
\begin{Theorem}\label{thm:daktheorythm}
The rational $(\fin)$-global unitary $K$-theory spectrum canonically splits as a sum of global Eilenberg-MacLane spectra:
\[
	\KU_{\bbQ}\simeq \bigvee_{n\in\mathbb Z}\Sigma^{2n}H(\RUrat)\text .
\]
\end{Theorem}

\subsection*{Organization}
The rest of this paper is devoted to the proof of Theorem \ref{thm:datheorem}.
In Section \ref{sec:outop} we replace the category of rational global functors with the equivalent but less complicated category of $\Outop$-modules.
Cyclic groups play an important role and we give a more economical description of $\Outop$-modules over these.
Section \ref{sec:ruout} is concerned with the identification of the representation ring in a form more suited to computing maps into it.
In the last section we reduce the computation to cyclic groups and put everything together to show Theorem \ref{thm:datheorem} and Corollary \ref{cor:dacorollary}.
\subsection*{Acknowledgements}
These results were obtained as part of my thesis written under the supervision of Stefan Schwede.
I would like to thank him and Markus Hausmann for helpful discussions related to this material. I also would like to thank Emanuele Dotto for reading a draft version.

\section{Global functors and \texorpdfstring{$\Outop$}{Outop}-modules}\label{sec:outop}

We start with a brief review of the definition of $(\fin)$-global functors:
They are indexed by the pre-additive \emph{$(\fin)$-global Burnside category} $\bbA$. Its objects are finite groups and the abelian group of morphisms $\bbA(G,K)$ is the Grothendieck group on isomorphism classes of finite $G$-free $(K,G)$-bisets with disjoint union as addition.
Composition is induced by the balanced product of finite bisets:
\[
\bbA(K,L)\times \bbA(G, K)\ra \bbA(G,L), \quad ([N], [M])\mapsto [N\times_K M]
\]

\begin{Definition}\label{def:globalfunctor}
The abelian category of \emph{global functors}, denoted by $\GF$,
is the category of additive functors from the Burnside category $\bbA$ to abelian groups.
\end{Definition}

More explicitly, a global functor $F$ consists of abelian groups $F(G)$ together with the following structure maps:
\begin{itemize}
\item For every subgroup inclusion $H\leq G$, a \emph{transfer map} $\tr_H^G:F(H)\ra F(G)$.
\item For every homomorphism $\alpha:K\ra G$, a \emph{restriction map} $\alpha^\ast:F(G)\ra F(K)$.
If $\alpha$ is surjective, we will also refer to this as an \emph{inflation map}.
\end{itemize}
These correspond to the distinguished morphisms $\tr_H^G=[G]\in \bbA(H,G)$ and $\res_{\alpha}=[\alpha^\ast G]\in\bbA(G,K)$ in the Burnside category. 
The definition in terms of bisets efficiently encodes the relations between these, which we do not spell out here.
\begin{Remark}
In \cite{global}, a more 'tautological' definition (from the homotopy theoretic point of view) of the global Burnside category is given in the generality of compact Lie groups.
Namely, the abelian group $\bbA(G,K)$ is defined to be the set of natural transformations $\operatorname{Nat}(\pi_0^G,\pi_0^K)$
between the equivariant homotopy groups of global homotopy types. There is a natural comparison functor sending transfers and restrictions to their topological counterparts and it is a computational fact that these definitions agree on finite groups. 
\end{Remark}

\begin{Example}
The complex representation rings $\RU(G)$ of finite groups $G$ form a global functor $\RU$ with structure maps
\[
	\bbA(G,K)\times\RU(G)\rightarrow \RU(K),\quad ([M], V)\mapsto \bbC\{M\}\tensor_GV,
\]
where $\bbC\{-\}$ is the $\bbC$-linearization of sets.
In particular, transfers are given by induction of representations and restrictions by restricting the group action.
We write $\RUrat$ for the rationalized functor with values $\RU(G)\tensor\bbQ$.
\end{Example}

\begin{Definition}
Let $\Out$ be the category of finite groups and conjugacy classes of surjective homomorphisms.
We denote by $\Outopmod$ the abelian category of contravariant functors from $\Out$ to abelian groups.
\end{Definition}
Rationally, the category of global functors can be greatly simplified.
Dividing out transfers from proper subgroups defines a functor
\[
	\tau:\GF\ra \Outopmod
\]
to the category of $\Outop$-modules. For a global functor $F$, the value of $\tau F$ at a finite group $G$ is thus given by
\[
	(\tau F)(G)=F(G)/\left(\sum_{H<G} \tr_H^G(F(H))\right)
\]
and the inflations of $F$ pass to the quotients since they commute with transfers.
\begin{Theorem}[{\cite[Thm 4.5.35]{global}}]
The above functor restricts to an equivalence
\[
	\tau:\GF_\bbQ\overset{\simeq}{\lra}\Outopmod_\bbQ
\]
between rational global functors and rational $\Outop$-modules.
\end{Theorem}


\begin{Definition}\label{def:outcyc}
For every $n\geq1$ we fix a cyclic group of order $n$ with a chosen generator $\tau_n$ and define $\Outopcyc\subset\Out$ be the full subcategory on these. For definiteness we take $C_n\subset \bbC^\times$ to be the $n$-th roots of unity and $\tau_n=e^{2\pi i/n}$.
\end{Definition}
There are preferred projections 
\begin{equation}\label{eq:daeqn}
p_{m,n}:C_{m}\twoheadrightarrow C_n,\quad \tau_m\mapsto\tau_n
\end{equation}
for all integers $n,m\geq 1$ such that $n$ divides $m$. Under the identification $C_k\cong \bbZ/k\bbZ$ they correspond to the unique surjections of rings,
which also restrict to surjections on units. The canonical isomorphisms
\(
(\bbZ/k\bbZ)^{\times}\cong\Out(C_k)=\Aut(C_k),\ l\mapsto (x\mapsto x^l)
\)
then provide surjections $(p_{m,n})_!:\Out(C_m)\twoheadrightarrow\Out(C_n)$ on automorphisms groups such that the square
\begin{equation}\label{eq:dasquare}
	\begin{tikzcd}
		C_m	\arrow[r, "p_{m,n}"]\arrow[d, "\phi"]	&	C_n\arrow[d, "p_!(\phi)"]	\\
		C_m	\arrow[r, "p_{m,n}"]				&	C_n
	\end{tikzcd}
\end{equation}
commutes for all $\phi\in \Out(C_m)$.
Furthermore, we note that any two epimorphisms between the same cyclic groups differ by a unique automorphism in the target. This implies the following description of $\Outopcyc$-modules:
\begin{Proposition}
An $\Outopcyc$-module $X$ is uniquely specified by the following data:
\begin{itemize}
\item
The collection of $\Out(C_n)$-modules $X(C_n)$ for $n\geq 1$.
\item
The collection of restriction maps $p_{m,n}^\ast:X(C_n)\rightarrow X(C_m)$ associated with the preferred projections $C_m\twoheadrightarrow C_n$ for $n\mid m$.
\end{itemize}
The restrictions $p_{m,n}^\ast$ have to be compatible with composition and $\Out(C_m)$-equivariant, where $\Out(C_m)$ acts on $X(C_n)$ via the canonical map
$\Out(C_m)\rightarrow\Out(C_n)$.

Similarly, a morphism $f:X\rightarrow Y$ of $\Outopcyc$-modules consists of $\Out(C_n)$-equivariant maps $f(C_n):X(C_n)\rightarrow Y(C_n)$ for $n\geq 1$, which commute with the distinguished restrictions.
\end{Proposition}

\section{The \texorpdfstring{$\Outop$}{Outop}-module associated with the global functor \texorpdfstring{$\RUrat$}{RUQ}}\label{sec:ruout}
Our goal in this section is to describe $\tau(\RUrat)$ in a form suitable for the later calculation of $\Ext$-groups.
We start by reviewing its identification in terms of cyclotomic field extensions, which appeared in an earlier version of \cite{global}.
By Artin's theorem (see e.g.\ \cite[II.9 Thm.\ 17]{serre}), every virtual representation is rationally induced from cyclic subgroups.
Since transfers in $\RUrat$ are given by induction of representations, $\tau(\RUrat)$ vanishes at non-cyclic groups. 
Let $X$ be the one-dimensional tautological representation of $C_n$. Then the value at the cyclic group $C_n$ of order $n$ is given by
\[
	\RUrat(C_n)=\bbQ[X]/(X^n-1)\text .
\]
Over the cyclic $p$-group $C_{p^k}$ a calculation with characters shows that dividing out transfers corresponds to dividing out the ideal
$(1+X^{p^{k-1}}+ \cdots + X^{(p-1)\cdot p^{k-1}})$
generated by the minimal polynomial  of the primitive $p^k$-th roots.
Hence we can identify
\[
	\tau(\RUrat)(C_{p^k})\cong \bbQ (\zeta_{p^k})
\]
as a cyclotomic field extension for all primes $p$ and non-negative integers $k$.
The group of automorphisms $\Out(C_{p^k})\cong(\mathbb{Z}/p^k\mathbb{Z})^{\times}$ acts as the Galois group
and restriction along the projection $C_{p^k}\twoheadrightarrow C_{p^{k-1}}$ is given by sending $\zeta_{p^{k-1}}$ to $\zeta_{p^k}^{p}$.

We recall that via the tensor product of representations the representation ring of a product of two groups is identified with the tensor product of the individual representation rings.
For coprime integers $n$ and $m$, the Chinese remainder theorem then implies that the canonical map of commutative rings induced by the two preferred restrictions
\[
	p_{nm,n}^\ast\cdot p_{nm,m}^\ast:\RUrat(C_n)\tensor\RUrat(C_m)\overset{\cong}{\lra}\RUrat(C_{nm})
\]
is an isomorphism. Since the proper transfers form an ideal, the objectwise ring structure in $\RUrat$
passes to the quotient and the restrictions of $\tau(\RUrat)$ become maps of commutative rings. A transfer in one of the tensor factors corresponds to a transfer from a subgroup of the form
$H\times C_m$ or $C_n\times H$ for a proper subgroup $H$. The groups of this form contain all maximal subgroups of the product because $n$ and $m$ are coprime.
Hence the induced multiplication map
\[
	\tau(\RUrat)(C_n)\tensor\tau(\RUrat)(C_m)\cong\tau(\RUrat)(C_{nm})
\]
is again an isomorphism. 
\begin{Remark}
The previous isomorphism allows us to identify
\[
	\tau(\RUrat)(C_n)\cong \bbQ (\zeta_n)
\]
as a cyclotomic field extension for all non-negative integers $n$ with $\Out(C_n)$ again acting as the Galois group.
However, we will only need to know this for cyclic $p$-groups.
\end{Remark}

To proceed we need a more convenient description of $\tau(\RUrat)$.
The Normal Basis Theorem in Galois theory states that as a representation of the Galois group, the extension field in a finite Galois extension is isomorphic to the regular representation over the subfield. In the case of cyclotomic field extensions these identifications can be made compatible
with the $\Outopcyc$-functoriality. 
\begin{Definition}
We denote by $\bbQ[\Out (-)]$ the $\Outopcyc$-module consisting of the collection of regular representations.
Its restriction maps
\[
	p^{\ast}:\bbQ[\Out(C_n)]\ra\bbQ[\Out(C_{nk})]
\]
are obtained from the surjections $\Out(C_{nk})\twoheadrightarrow \Out(C_n)$ (\ref{eq:dasquare}) by summation over the fibers,
which is automatically $\Out(C_{nk})$-equivariant.
In additive notation, it sends a basis element $ j\in  (\bbZ/n\bbZ)^\times\cong\Out(C_n)$ to the sum over all basis elements $\tilde j$ such that $\tilde j\in (\bbZ/nk\bbZ)^\times$ reduces to $j$ mod $n$.  
\end{Definition}
Since an equivariant map out of the regular representation is uniquely determined by its value on the identity element,
a morphism $\bbQ[\Out(-)]\rightarrow X$ of $\Outopcyc$-modules is classified by elements $x_n\in X(C_n)$ (subject to the relations expressing naturality).
In the following lemma $\Outopcycp$ denotes the full subcategory on cyclic $p$-groups.
\begin{Lemma}\label{lemma:outassemble}
Let $R$ be a commutative $\Outopcyc$-ring \emph(i.e.\ an $\Outopcyc$-diagram of commutative rings\emph) and suppose that we are given for each prime $p$ a morphism
\(
	\phi^p:\bbQ[\Out(-)]\rightarrow R
\)
of the underlying $\Outopcycp$-modules, classified by elements $x_{p^k}\in R(C_{p^k})$ with $x_1=1$.
Then the $\phi^p$ extend to a morphism
\[
	\phi:\bbQ[\Out(-)]\rightarrow R
\]
of $\Outopcyc$-modules whose classifying elements are determined by the relation $x_{nm}=x_n\cdot x_m\in R(C_{nm})$ for coprime integers $n$ and $m$.

Moreover, if we start out with isomorphisms $\phi^p$ and the multiplication maps $R(C_m)\tensor R(C_n)\ra R(C_{nm})$ are also isomorphisms for coprime integers, then $\phi$ is an isomorphism.

\end{Lemma}
\begin{proof}
Given coprime integers $n$ and $m$, suppose that we have already constructed partial morphisms for these, i.e.\ maps $\phi_k:\bbQ[\Out(C_k)]\ra R(C_k)$ for all $k$ dividing either $n$ or $m$ that commute with restrictions
and such that $\phi_1=1$ is the inclusion of the unit element in $R(C_1)$.
The morphism $\phi_{nm}$ is now defined as the composite 
\begin{equation*}
\bbQ[\Out(C_{nm})]\cong \bbQ[\Out(C_n)]\tensor\bbQ[\Out(C_m)]\xrightarrow{\phi_n\tensor\phi_m}R(C_n)\tensor R(C_m)\ra R(C_{nm})\text{,}
\end{equation*}
where the first map is induced by the canonical decomposition $\bbZ/nm\bbZ\cong \bbZ/n\bbZ\times \bbZ/m\bbZ$.
This definition does not depend on the order of $n$ and $m$ because of the commutativity of $R$. Furthermore the condition $\phi_1=1$ ensures that in the case $n=1$ or $m=1$ we recover the original map. Replacing $n$ and $m$ by one of their divisors respectively in the above composite then defines $\phi$ for all divisors of $nm$.
This is compatible with restrictions because the initially defined maps are and so we have defined a partial morphism for $nm$.
Clearly, $\phi_{nm}$ is an isomorphism under the additional hypothesis.

Finally, we remark that this construction is independent of the decomposition into a product of two coprime integers. Considering the prime factorization of an integer $n$, we see that $\phi_{n}$ is just given by the analogous construction for several tensor factors applied to the initially defined maps for prime powers.
\end{proof}

\begin{Proposition}\label{prop:outru}
There is an isomorphism of $\Out_{cyc}^{op}$-modules 
\[
	\bbQ[\Out (-)]\overset{\cong}{\lra} \tau(\RUrat)\text{,}
\]
which is classified by the elements $\zeta_{p^k}+\zeta_{p^k}^p+\cdots+ \zeta_{p^k}^{p^{k-1}} \in \bbQ(\zeta_{p^k})$.
\end{Proposition}
\begin{proof}
The isomorphism is first constructed over cyclic $p$-groups where one can write down a normal basis explicitly. We abbreviate $X=\zeta_{p^k}$ in the following.
A direct computation shows that the element $X+X^p+\cdots+ X^{(p^{k-1})} \in \bbQ(\zeta_{p^k})$ generates the cyclotomic field extension as an $\Out(C_{p^k})$-module.
It follows that the $\Out(C_{p^k})$-equivariant map
\begin{equation*}
\phi_{p^k}:\bbQ[\Out(C_{p^k})]\xrightarrow{\frac{-1}{p^{k-1}}(X+X^p+\cdots+ X^{(p^{k-1})})}\bbQ (\zeta_{p^k})
\end{equation*}
is an isomorphism because both sides have the same rank. In the degenerate case $k=0$ we take for $\phi_1$ the canonical identification sending the identity to $1$. 

We now check compatibility with restrictions. Under the summation of fibers map the unit element $1\in \Out(C_{p^k})$
maps to the sum over the elements $1+lp^{k} \in  \Out(C_{p^{k+1}})$ for $0\leq l \leq p-1$, where we have used additive notation. Letting these act as elements in the Galois group of $\bbQ(\zeta_{p^{k+1}})$ and writing $Y=\zeta_{p^{k+1}}$, we get the equality
\begin{equation*} (1+lp^{k} )^{\ast}(Y+Y^p+\cdots+Y^{p^{k}})=Y\cdot Y^{lp^k}+Y^p+Y^{p^2}+\cdots+Y^{p^{k}}\text{.}
\end{equation*}
After summation over $l$ this becomes $p\cdot (Y^p+Y^{p^2}+\cdots+Y^{p^k})$ because the terms $Y^{lp^k}$ add up to the minimal polynomial
and hence do not contribute. But this is just the restriction of $p\cdot (X+X^p+\cdots+X^{(p^{k-1})})$. So the scaling  ensures that for varying $k\geq 0$, the $\phi_{p^k}$ commute with the restriction maps. The sign is needed for the case $k=0$.

By the previous lemma, these maps assemble to an isomorphism of $\Outop$-modules.

\end{proof}

\section{Computation of \texorpdfstring{$\Ext$}{Ext}-groups}

As we recalled in the previous section, $\tau(\RU)$ is concentrated at cyclic groups. 
The following lemma shows that we thus may restrict to $\Outopcyc$-modules.
\begin{Lemma}\label{lemma:restrictcyclic}
Restriction along the inclusion \(	\iota:\Out_{\operatorname{cyc}}\hookrightarrow\Out	\) induces isomorphisms
\[
	\Ext_{\Outop\hyph\module_\bbQ}^n(X, \tau(\RUrat))\cong \Ext_{\Outopcyc\hyph\module_\bbQ}^n(\iota^{\ast}X, \iota^\ast(\tau(\RUrat)))
\]
on $\Ext$-groups for all $n\geq0$.
\end{Lemma}
\begin{proof}
The right Kan extension $\iota_{\ast}$ simply extends by $0$ to non-cyclic groups. Since $\tau(\RUrat)$ is concentrated at cyclic groups,
it is right-induced in the sense that the unit map $\tau(\RUrat)\ra\iota_{\ast}\iota^{\ast}(\tau(\RUrat))$ is an isomorphism.
Both restriction and right Kan extension are exact functors and so $\iota^{\ast}$ preserves projective resolutions. The claim now follows by adjointness.
\end{proof}

\begin{Remark}\label{remark:infinitehomological}
Even though we work rationally, the derived category $\mathcal D(\Outopmod_\bbQ)$ does \emph{not} split and all higher extensions can occur.
For example, the $\Outop$-module $R_e\bbQ$ that consists of a copy of $\bbQ$ at the trivial group does not admit a finite projective resolution.
It suffices to show this over cyclic groups since the restriction functor along the inclusion of the full subcategory $\Out_{\cyc}\subset\Out$ preserves projective resolutions.
Let $F_n=\bbQ\{\Out(-,C_n)\}/\Out(C_n)$ be the 'semi-free' $\Outopcyc$-module generated by $\bbQ$ at the cyclic group $C_n$,
characterized by the natural isomorphism 
\[
	\Hom_{\Outopcyc\hyph\module}(F_n, X)\cong X(C_n)^{\Out(C_n)}
\]
that corresponds to the universal element $\Id_{C_n}$. Then $F_n$ consists of a single copy of $\bbQ$ at every $C_m$ such that $n$ divides $m$,
with identities as structure maps, and vanishes otherwise. A projective resolution of $R_e\bbQ$ is defined as follows:
\[
R_e\bbQ\twoheadleftarrow F_e\longleftarrow \bigoplus_{p}F_p\longleftarrow\bigoplus_{p_1<p_2}F_{p_1p_2}\longleftarrow\cdots\longleftarrow\bigoplus_{p_1<p_2<\cdots<p_n}F_{p_1p_2\cdots p_n}\longleftarrow\cdots
\]
The sum $P_n=\oplus_{p_1<\cdots<p_n}F_{p_1\cdots p_n}$ is indexed by all $n$-element sets of primes and the differential $d_n:P_n\ra P_{n-1}$ is determined by the formula
\[
d(e^{p_1\cdots p_n}_{p_1<\ldots<p_n})= \sum^{n}_{i=1}(-1)^ie^{p_1\cdots\widehat{p_i}\cdots p_n}_{p_1<\cdots<\widehat{p_i}<\ldots<p_n},
\]
where $e^m_{p_1<\cdots<p_n} \in P_n(C_m)$ denotes the generator of the summand indexed by $p_1,\ldots, p_n$,
which is given by the class of (any) surjection $[C_m\twoheadrightarrow C_{p_1\cdots p_n}]\in F_{p_1\cdots p_n}(C_m)$. 
For every integer $m\geq 1$, we construct a chain contraction $h$ of the complex $P_\ast(C_m)$ of $\bbQ\hyph$vector spaces:
Let $p_1<\cdots<p_{\omega(m)}$ be the prime factors of $m$.
Given a sequence
\[
1\leq\alpha(1)<\cdots<\alpha(n)\leq \omega(m),
\]
we set $e^m_\alpha=e^m_{p_{\alpha(1)}<\cdots<p_{\alpha(n)}}$,
and for $k\not\in \operatorname{im}(\alpha)$ we let $e^m_{\alpha,k}=e^m_{\tilde\alpha}$ be the element corresponding to the sequence
\(
	\tilde\alpha=(\cdots<\alpha(j-1)<k<\alpha(j)<\cdots)
\)
obtained by adding $k$ to it. The maps $h_n:P_n(C_m)\ra P_{n+1}(C_m)$ are then defined by
\[
h_n(e^m_{\alpha})=\frac{1}{\omega(m)}\sum_{k\not\in \operatorname{im}(\alpha)}(-1)^{j(\alpha, k)}e^m_{\alpha, k}
\]
and one checks that this indeed yields a chain contraction, showing that $P_\ast\twoheadrightarrow R_e\bbQ$ is a resolution.
We claim that the projections $\xi_n:P_n\twoheadrightarrow\operatorname{coker}(d_{n+1})$ define non-trivial elements
$[\xi_n]\in \Ext^n_{\Outopcyc\hyph\module_\bbQ}(R_e\bbQ, \operatorname{coker}(d_{n+1}))$:
The 'universal' cocycle $\xi_n$ is non-trivial by a dimension count and cannot be a coboundary since there are no non-zero morphisms to the cokernel.
The group $\Hom_{\Outopcyc\hyph\module_\bbQ}(P_{n-1}, \operatorname{coker}(d_{n+1}))=0$ is trivial because $P_n$ (and hence the cokernel) vanishes at all groups at which $P_{n-1}$ is generated as an $\Outopcyc$-module.
\end{Remark}

The identification $\tau(\RUrat)\cong\bbQ[\Out (-)]$ from the previous section allows us to describe maps into it as an inverse limit over the poset of natural numbers with partial order given by the divisibility relation:
\begin{Definition}\label{def:inversesystem}
If $X$ is an $\Outopcyc$-module, we denote by $X^{\vee}$ the inverse system obtained after forming $\bbQ$-linear duals in each level and forgetting all group actions. The structure maps are defined by precomposition with the preferred restriction maps (\ref{eq:daeqn}).
\end{Definition}

\begin{Proposition}
There is a natural isomorphism
\begin{equation*}
\Hom_{\Outopcyc\hyph\module_\bbQ}(X, \bbQ[\Out(-)])\cong \varprojlim_{(\mathbb{N},|)}(X^{\vee})\text,
\end{equation*}
which sends $\phi:X\ra\bbQ[\Out(-)]$ to the collection of linear forms
\[
	(\tilde\phi_n:X(C_n)\overset{\phi_n}{\lra}\bbQ[\Out(C_n)]\overset{\operatorname{pr}_e}{\lra}\bbQ)_{n\in \bbN}.
\]
\end{Proposition}
\begin{proof}
For any finite group, giving an equivariant map into the regular representation is equivalent to specifying the linear form obtained by projecting to the summand of the neutral element. Hence a map $\phi:X\ra \bbQ[\Out(-)]$ is uniquely  determined by the linear forms $\tilde\phi_n \in X(C_n)^{\vee}$. The condition that $\phi$ is a natural transformation
translates into the condition that these linear forms restrict to each other because the summation over the fibers map commutes with the projection to the summand of the neutral element in $\bbQ[\Out(-)]$. This means that the collection $\{\phi_n\}$ forms an element in the inverse limit.
\end{proof}

The functor $(-)^{\vee}$ is exact and turns sums into products. It sends enough projectives to injectives,
e.g.\ those of the form $F_nV=V\tensor_{\Out(C_n)} \bbQ\{\Out(-,C_n)\}$ for an $\Out(C_n)$-representation V.
Indeed, $F_nV$ is just the constant functor $V$ at cyclic groups with order divisible by $n$ and it vanishes elsewhere. Hence we can also identify the higher derived functors:
\begin{Corollary}\label{cor:extout}
The map of the previous proposition induces natural isomorphisms
\begin{equation*}
\Ext^{n}_{\Outopcyc\hyph\module_\bbQ}(X, \bbQ[\Out(-)])\cong {\varprojlim_{(\mathbb{N},\mid)}}^n(X^{\vee})
\end{equation*}
for all $n\geq 0$.
\end{Corollary}

We can now show the main theorem and its corollary:
\begin{proof}[Proof of Theorem \ref{thm:datheorem}]
Combining the equivalence $\GF_\bbQ\simeq\Outopmod_\bbQ$, Lemma \ref{lemma:restrictcyclic}, Proposition \ref{prop:outru},
and Corollary \ref{cor:extout} yields the following chain of natural isomorphisms:
\begin{align*}
\Ext^{k}_{\GF}(F, \RUrat)	\cong\Ext^k_{\Outop\hyph\module_\bbQ}(\tau F, \tau(\RUrat))					
					&\cong\Ext^k_{\Outop_{\operatorname{cyc}}\hyph\module_\bbQ}(\tau F,\tau(\RUrat))			\\
					&\cong\Ext^k_{\Outop_{\operatorname{cyc}}\hyph\module_\bbQ}(\tau F, \bbQ[\Out(-)])			\\
					&\cong{\varprojlim_{(\bbN,\mid)}}^k(\tau F)^{\vee}		\\
					&\cong{\varprojlim_{(\bbN,\mid)}}^k(F(C_n)^{\vee}_{\tr})\text{.}\qedhere
\end{align*}
\end{proof}
Finally, we observe that $(\mathbb{N},\mid)$ contains the sequential poset $\mathbb{N}$ as a cofinal subset via the factorials.
Restriction to this subset is exact and the left Kan extension takes an inverse system $\{X_{k!}\}_{k\in\mathbb{N}}$
defined over the factorials to the inverse system defined for all integers $n$ by $X_n=X_{k!}$, where $k$ is minimal such that $n$ divides $k!$.
Hence it is also exact and we only have to compute sequential limits.
\begin{proof}[Proof of Corollary \ref{cor:dacorollary}]
There is only a potential ${\varprojlim}^1$-term for sequential systems and this gives the first part:
\(
	\Ext^{n}_{\GF}(F, \RUrat)=0\text{ for }n\geq 2\text{.}
\)
The structure maps of $X=\bbQ[\Out(-)]$ are injective since they are defined by summation over the fibers of surjective maps.
So the inverse system associated with $X$ only consists of surjective maps and this implies the second part:
\(
\Ext^{1}_{\GF}(\RUrat, \RUrat)=0\text{.}
\)
\end{proof}

\bibliographystyle{alpha}
\bibliography{references}

\end{document}